\documentclass[reqno, 12pt]{amsart}

\numberwithin{equation}{section}

\usepackage{amsfonts,amssymb,amsmath,amsthm}
\usepackage{enumerate}
\usepackage{bbm}
\usepackage{times}
\usepackage{amsfonts}
\usepackage{amssymb}
\usepackage{bbm}
\usepackage{times}
\usepackage{amsfonts}
\usepackage{amssymb}
\usepackage{bbm}
\usepackage{times}

\usepackage{enumerate}
\usepackage{amsmath}
\usepackage{amsthm}
\usepackage{color}

\newtheorem{theor}{Theorem}[section]

\newtheorem{limma}[theor]{Lemma}

\newtheorem{prop}[theor]{Proposition}

\newcounter{other}            % Questions get letters
\newtheorem{otherth}[other]{Theorem}              % Other papers' theorems
\newtheorem{otherp}[other]{ Proposition}% Other papers'propositions
\newtheorem{otherl}[other]{ Lemma}        % Other papers' lemmas

\def \B{\mathcal B}
\def \Q{\mathcal Q}
\def \Cu{\mathcal{C}_\mu}
\def \D{\mathbb{D}}
\def \c{\mathbb{C}}

\def \T{\mathbb{T}}

\def \f{\frac}
\def \C {\mathcal C}
\def \qp {\mathcal{Q}_p}
\def \Cus{\mathcal{C}_{\mu, s}}

\begin{document}
\title[The range of a Ces\`aro-like operator acting on $H^\infty$ ]
{Carleson measures and the range of a Ces\`aro-like operator acting on $H^\infty$}

\author{Guanlong Bao, Fangmei Sun and Hasi Wulan}
\address{Guanlong Bao\\
Department of Mathematics\\
    Shantou University\\
    Shantou, Guangdong 515063, China}
\email{glbao@stu.edu.cn}

\address{Fangmei Sun\\
Department of Mathematics\\
    Shantou University\\
    Shantou, Guangdong 515063, China}
\email{18fmsun@stu.edu.cn}

\address{Hasi Wulan\\
Department of Mathematics\\
    Shantou University\\
    Shantou, Guangdong 515063, China}
\email{wulan@stu.edu.cn}

\thanks{The work was supported  by NNSF of China (No. 11720101003) and NSF of Guangdong Province (No. 2022A1515012117).}

\begin{abstract}
In this paper, by describing characterizations of Carleson type measures on $[0,1)$, we determine the range of a Ces\`aro-like operator acting on $H^\infty$. A special case of our result gives an answer to a question posed by P. Galanopoulos, D. Girela and N. Merch\'an recently.

\end{abstract}

\maketitle

\section{Introduction}

 Let $\D$ be the open unit disk in the complex plane $\c$.  Denote by $H(\D)$  the space of functions analytic in $\D$.
For $f(z)=\sum_{n=0}^\infty a_nz^n$ in  $H(\D)$, the Ces\`aro operator $\C$ is defined  by
 $$
\C (f)(z)=\sum_{n=0}^\infty\left(\f{1}{n+1}\sum_{k=0}^n a_k\right)z^n, \quad z\in\D.
$$
See \cite{BWY, DS,  EX, M, S1, S2} for the investigation of the Ces\`aro operator acting on some analytic function spaces.

Recently, P. Galanopoulos, D. Girela and N. Merch\'an \cite{GGM} introduced a Ces\`aro-like operator $\Cu$ on $H(\D)$.
For nonnegative integer $n$, let $\mu_n$ be the moment of order $n$ of a finite positive Borel measure  $\mu$ on  $[0, 1)$; that is,
$$
\mu_n=\int_{[0, 1)} t^{n}d\mu(t).
$$
For $f(z)=\sum_{n=0}^\infty a_nz^n$ belonging to   $H(\D)$, the Ces\`aro-like operator $\Cu$ is defined by
$$
\Cu (f)(z)=\sum^\infty_{n=0}\left(\mu_n\sum^n_{k=0}a_k\right)z^n, \quad z\in\D.
$$
If $d\mu(t)=dt$, then $\Cu=\C$. In \cite{GGM}, the authors studied the action of $\Cu$ on distinct spaces of analytic functions.

We also need to  recall some function spaces. For $0<p<\infty$, $H^p$ denotes the classical Hardy space \cite{D} of  those functions  $f\in H(\D)$ for which
$$
\sup_{0<r<1} M_p(r, f)<\infty,
$$
where
$$
M_p(r, f)= \left(\f{1}{2\pi}\int_0^{2\pi}|f(re^{i\theta})|^p d\theta \right)^{1/p}.
$$
As usual, denote by $H^\infty$ the space of bounded analytic functions in  $\D$. It is well known that $H^\infty$ is a proper subset of the Bloch space $\B$ which consists of those functions $f\in H(\D)$ satisfying
$$
\|f\|_\B=\sup_{z\in \D}(1-|z|^2)|f'(z)|<\infty.
$$

Denote by $\text{Aut}(\D)$ the group of M\"obius maps of $\D$, namely,
$$
\text{Aut}(\D)=\{e^{i\theta}\sigma_a:\  \ a\in \D \ \text{and} \ \theta \ \ \text{is real}\},
$$
where
$$
\sigma_a(z)=\frac{a-z}{1-\overline{a}z}, \qquad z\in \D.
$$
  In 1995  R. Aulaskari,  J. Xiao and  R. Zhao \cite{AXZ} introduced  $\qp$ spaces.  For
$0\leq p<\infty$, a function $f$ analytic in $\D$   belongs to $\qp$ if
$$
\|f\|_{\mathcal{Q}_p}^2=\sup_{w\in \D} \int_\D |f'(z)|^2(1-|\sigma_w(z)|^2)^p dA(z)<\infty,
$$
where  $dA$ is the area measure on $\c$ normalized so that $A(\D)=1$. $\qp$ spaces  are M\"obius invariant in the sense that
$$
\|f\|_{\qp}=\|f\circ \phi\|_{\qp}
$$
for every $f\in \qp$ and $\phi \in \text{Aut}(\D)$. It was shown in  \cite{X1} that $\Q_2$ coincides with the Bloch space $\B$.
This result was extended in  \cite{AL} by showing that $\qp=\B$ for all $1<p<\infty$.
The space $\Q_1$ coincides with $BMOA$, the set of analytic functions in  $\D$ with boundary values of bounded mean oscillation (see \cite{B, Gir}). The space $\Q_0$ is the Dirichlet space $\mathcal D$.  For $0<p<1$, the space $\qp$ is a proper subset of $BMOA$ and has many interesting
properties.  See J. Xiao's monographs  \cite{X2, X3} for the theory of  $\qp$ spaces.

For  $1\leq p<\infty$ and $0<\alpha\leq 1$, the mean Lipschitz space $\Lambda^p_\alpha$ is the set of those functions $f\in H(\D)$ with a non-tangential limit  almost everywhere such that $\omega_p(t, f)=O(t^\alpha)$ as $t\to 0$. Here $\omega_p(\cdot, f)$ is the integral modulus of continuity of order $p$ of the function $f(e^{i\theta})$. It is well known (cf.\cite[Chapter 5]{D}) that $\Lambda^p_\alpha$ is a subset of $H^p$ and $\Lambda^p_\alpha$ consists of those functions $f\in H(\D)$ satisfying
$$
\|f\|_{\Lambda^p_\alpha}=\sup_{0<r<1}(1-r)^{1-\alpha}M_p(r, f')<\infty.
$$
Among these spaces, the spaces $\Lambda^p_{1/p}$  are of special interest.   $\Lambda^p_{1/p}$ spaces  increase with $p\in (1, \infty)$ in the sense of inclusion and they  are contained in $BMOA$ (cf. \cite{BSS}).
By Theorem 1.4 in \cite{ASX}, $\Lambda^p_{1/p}\subseteq \Q_q$ when $1\leq p<2/(1-q)$ and $0<q<1$. In particular,  $\Lambda^2_{1/2}\subseteq \Q_q \subseteq \B$ for all $0<q<\infty$.

Given an arc $I$ of the unit circle $\T$ with arclength $|I|$ (normalized  such that $|\T|=1$),  the
 Carleson box  $S(I)$ is given by
$$
S(I)=\{r\zeta \in \D: 1-|I|<r<1, \ \zeta\in I\}.
$$
  For $0<s<\infty$,  a positive   Borel measure $\nu$ on $\D$ is said to be an  $s$-Carleson measure  if
$$
\sup_{I\subseteq\T}\frac{\nu(S(I))}{|I|^s}<\infty.
$$
If $\nu$ is a $1$-Carleson measure, we write   that $\nu$ is a Carleson measure  characterizing $H^p\subseteq L^p(d\nu)$ (cf. \cite{D}).
A positive Borel measure $\mu$ on [0, 1) can be seen as a Borel measure on $\D$ by identifying it
with the measure $\tilde{\mu}$ defined by
$$
\tilde{\mu}(E)=\mu(E \cap [0, 1)),
$$
for any Borel subset $E$ of $\D$. Thus $\mu$ is an $s$-Carleson measure on $[0,1)$ if there is a positive constant $C$ such that
$$
\mu([t, 1)) \leq C (1-t)^s
$$
for all $t\in [0, 1)$. We refer to \cite{BYZ} for the investigation of this kind of  measures associated with Hankel measures.

It is known that the Ces\`aro operator $\C$  is bounded on $H^p$ for all $0<p<\infty$ (cf. \cite{M, S1, S2}) but this is not true on $H^\infty$. In fact,  N. Danikas and A. Siskakis \cite{DS} gave   that
$\mathcal{C }(H^\infty)\nsubseteq H^\infty$ but $\mathcal{C }(H^\infty)\subseteq BMOA$. Later M. Ess\'en and J. Xiao \cite{EX} proved that $\mathcal{C }(H^\infty)\subsetneqq \qp$ \ for \  $0<p<1$. Recently, the relation between $\mathcal{C }(H^\infty)$ and a class of M\"obius invariant function spaces was considered in \cite{BWY}.

It is quite natural to study $\Cu(H^\infty)$. In \cite{GGM} the authors characterized positive Borel measures $\mu$ such that $\Cu(H^\infty)\subseteq H^\infty$ and proved that $\Cu(H^\infty)\subseteq \B$ if and only if $\mu$ is a Carleson measure. Moreover, they showed that if $\Cu(H^\infty)\subseteq BMOA$, then $\mu$ is a Carleson measure. In \cite[p. 20]{GGM}, the authors asked whether or not $\mu$ being a Carleson measure implies that $\Cu(H^\infty)\subseteq BMOA$.  In this paper,
by giving some descriptions of
$s$-Carleson measures on $[0,1)$,
for $0<p<2$,  we show that $\Cu(H^\infty)\subseteq \qp$ if and only if $\mu$ is a Carleson measure, which giving an affirmative answer to their  question.  We also  consider another Ces\`aro-like operator $\C_{\mu, s}$ and describe the embedding  $\Cus(H^\infty)\subseteq X$ in terms of $s$-Carleson measures, where $X$ is
between $\Lambda^p_{1/p}$ and $\B$ for $\max\{1, 1/s\}<p<\infty$.

Throughout  this paper, the symbol $A\thickapprox B$ means that $A\lesssim
B\lesssim A$. We say that $A\lesssim B$ if there exists a positive
constant $C$ such that $A\leq CB$.

\section{Positive Borel measures on [0, 1)  as Carleson type measures }

In this section, we give some characterizations of positive Borel measures on [0, 1)  as Carleson type measures.

The following  description of Carleson type measures (cf. \cite{Bla} ) is well known.
\begin{otherl}\label{S-CM}
Suppose $s>0$, $t>0$ and $\mu$ is a positive  Borel measure on $\D$. Then $\mu$ is an $s$-Carleson measure if and only if
\begin{equation}\label{sCMformula}
\sup_{a\in \D}\int_{\D} \frac{(1-|a|^2)^t}{|1-\overline{a}w|^{s+t}}d\mu(w)<\infty.
\end{equation}
\end{otherl}

For Carleson type  measures  on   [0, 1),  we can obtain some  descriptions that are  different from Lemma \ref{S-CM}.  Now we give  the first main result in this section.

\begin{prop}\label{newCM1}
Suppose  $0<t<\infty$, $0\leq  r<s<\infty$ and  $\mu$ is a finite positive  Borel measure on $[0,1)$. Then the following conditions are equivalent:
\begin{enumerate}
  \item [(i)] $\mu$ is an $s$-Carleson measure;
   \item [(ii)]  \begin{equation}\label{1formulaCM}
\sup_{a\in\D}\int_{[0,1)}\frac{(1-|a|)^t}{(1-x)^{r}(1-|a|x)^{s+t-r}}d\mu(x)<\infty;
\end{equation}
  \item [(iii)] \begin{equation}\label{2formulaCM}
\sup_{a\in\D}\int_{[0,1)}\frac{(1-|a|)^t}{(1-x)^{r}|1-ax|^{s+t-r}}d\mu(x)<\infty.
\end{equation}
\end{enumerate}
\end{prop}

\begin{proof}
$(i)\Rightarrow (ii)$. Let $\mu$ be  an $s$-Carleson measure.  Fix $a\in \D$ with $|a|\leq 1/2$. If  $r=0$, the desired result holds.  For $0<r<s$, using a well-known formula about the distribution function(\cite[p.20 ]{Gar}), we get
\begin{align}\label{bu1}
&\int_{[0,1)}\frac{(1-|a|)^t}{(1-x)^{r}(1-|a|x)^{s+t-r}}d\mu(x)\nonumber \\
\thickapprox & \int_{[0,1)}\left(\frac{1}{1-x}\right)^rd\mu(x) \nonumber \\
\thickapprox & r \int_0^\infty \lambda^{r-1} \mu(\{x\in [0, 1): 1-\frac{1}{\lambda}<x \})d\lambda \nonumber \\
\lesssim & \int_0^1 \lambda^{r-1} \mu([0, 1))d\lambda + \int_1^\infty \lambda^{r-1} \mu([1-\frac{1}{\lambda}, 1))d\lambda \nonumber \\
 \lesssim & 1+ \int_1^\infty \lambda^{r-s-1} d\lambda \lesssim1.
\end{align}
Fix $a\in\D$ with $|a|>1/2$ and let
\begin{align*}
	S_n(a)=\{x\in[0,1): 1-2^n(1-|a|)\leq x<1\}, \ \ n=1, 2, \cdots .
\end{align*}
Let $n_a$ be the minimal integer such that $1-2^{n_a}(1-a)\leq 0$. Then $S_n(a)=[0, 1)$ when $n\geq n_a$.
If   $ x\in S_1(a)$, then
 \begin{equation}\label{301}
1-|a| \leq  1-|a|x.
 \end{equation}
Also, for $2\leq n\leq  n_a$ and $x\in S_n(a)\backslash S_{n-1}(a)$, we have
\begin{equation}\label{302}
1-|a|x \geq |a|-x \geq |a|-(1-2^{n-1}(1-|a|))=(2^{n-1}-1)(1-|a|).
\end{equation}
We write
\begin{align*}
	& \int_{[0,1)}\frac{(1-|a|)^t}{(1-x)^r(1-|a|x)^{s+t-r}}d\mu(x)\\
	=&\int_{S_1(a)}\frac{(1-|a|)^t}{(1-x)^r(1-|a|x)^{s+t-r}}d\mu(x)\\
&+\sum^{n_a}_{n=2}\int_{S_n(a)\backslash S_{n-1}(a)}\frac{(1-|a|)^t}{(1-x)^r(1-|a|x)^{s+t-r}}d\mu(x)\\
	=: &J_1(a)+J_2(a).
\end{align*}
If $r=0$, bearing in mind that (\ref{301}), (\ref{302}) and $\mu$ is an $s$-Carleson measure, it is easy to check that  $J_i(a)\lesssim 1$ for $i=1, 2$. Now consider  $0<t<\infty$ and  $0< r<s<\infty$. Using  (\ref{301})   and   some estimates similar to (\ref{bu1}), we have
\begin{align*}
	J_1(a) \lesssim (1-|a|)^{r-s}\int_{S_1(a)}\left(\frac{1}{1-x}\right)^rd\mu(x)
\lesssim   1.
\end{align*}
Note that  (\ref{302}),  $0<t<\infty$,  $0< r<s<\infty$ and  $\mu$ is an $s$-Carleson measure.
Then
{\small {\small
\begin{align*}
	&J_2(a)\\
\lesssim& \sum^{n_a}_{n=2}\frac{(1-|a|)^{r-s}}{2^{n(s+t-r)}}\int_{S_n(a)\backslash S_{n-1}(a)}\left(\frac{1}{1-x}\right)^rd\mu(x)\\
	\lesssim& \sum^{n_a}_{n=2}\frac{(1-|a|)^{r-s}}{2^{n(s+t-r)}}\int_0^\infty\lambda^{r-1}\mu\big(\big\{x\in[1-2^n(1-|a|),1): 1-\frac{1}{\lambda}<x\big\}\big)d\lambda\\
	\thickapprox& \sum^{n_a}_{n=2}\frac{(1-|a|)^{r-s}}{2^{n(s+t-r)}}\bigg(\int_0^{\frac{1}{2^n(1-|a|)}}\lambda^{r-1}\mu\big([1-2^n(1-|a|),1)\big)d\lambda\\
	&+\int_{\frac{1}{2^n(1-|a|)}}^\infty\lambda^{r-1}\mu\big(\big[1-\frac{1}{\lambda},1\big)\big)d\lambda\bigg)\\
	\lesssim& \sum^{n_a}_{n=2}\frac{(1-|a|)^{r-s}}{2^{n(s+t-r)}}\bigg(2^{ns}(1-|a|)^s\int_0^{\frac{1}{2^n(1-|a|)}}\lambda^{r-1}d\lambda
	+\int_{\frac{1}{2^n(1-|a|)}}^\infty\lambda^{r-1-s}d\lambda\bigg)\\
\thickapprox&  \sum^{n_a}_{n=2} \frac{1}{2^{tn}}<\infty.
	\end{align*}}}
Consequently,
$$
\sup_{a\in \D}\int_{[0,1)}\frac{(1-|a|)^t}{(1-x)^r(1-|a|x)^{s+t-r}}d\mu(x)<\infty.
$$

The implication of $(ii)\Rightarrow (iii)$ is clear.

$(iii)\Rightarrow (i)$.  For $r\geq 0$, it is clear that
\begin{align*}
\int_{[0,1)}\frac{(1-|a|)^t}{(1-x)^{r}|1-ax|^{s+t-r}}d\mu(x)\geq \int_{[0,1)}\frac{(1-|a|)^t}{|1-ax|^{s+t}}d\mu(x)
\end{align*}
for all $a\in \D$. Combining this with Lemma \ref{S-CM}, we see that if (\ref{2formulaCM}) holds, then $\mu$ is an $s$-Carleson measure.
\end{proof}

\noindent {\bf  Remark 1.}\ \ The condition $0\leq  r<s<\infty$ in Proposition \ref{newCM1} can not be changed to $r\geq s>0$. For example, let $d\mu_1(x)=(1-x)^{s-1}dx$, $x\in [0, 1)$. Then $\mu_1$ is an $s$-Carleson measure but for $r\geq s>0$,
\begin{align*}
&\sup_{a\in\D}\int_{[0,1)}\frac{(1-|a|)^t}{(1-x)^{r}|1-ax|^{s+t-r}}d\mu_1(x)\\
\geq & \int_0^1 (1-x)^{s-1-r}dx=+\infty.
\end{align*}

\noindent {\bf  Remark 2.}\ \ $\mu$ supported  on $[0,1)$ is essential in Proposition \ref{newCM1}. For example,  consider $0<t<1$, $0<  r<s<1$ and $s=r+t$.  Set $d\mu_2(w)=|f'(w)|^2(1-|w|^2)^sdA(w)$, $w\in \D$, where $f\in \Q_s\setminus \Q_t$. Note that for $0<p<\infty$ and  $g\in H(\D)$,  $|g'(w)|^2(1-|w|^2)^pdA(w)$ is a $p$-Carleson measure  if and only if $g\in\Q_p$ (cf. \cite{X2}). Hence  $d\mu_2$ is an $s$-Carleson measure. But
\begin{align*}
	&\sup_{a\in\D}\int_{\D}\frac{(1-|a|)^t}{(1-|w|)^{r}|1-a\overline{w}|^{s+t-r}}d\mu_2(w)\\
	&=\sup_{a\in\D}\int_{\D}|f'(w)|^2\frac{(1-|a|)^t(1-|w|)^{s-r}}{|1-a\overline{w}|^{s+t-r}}dA(w)\\
	&\thickapprox\sup_{a\in\D}\int_{\D}|f'(w)|^2(1-|\sigma_a(w)|^2)^tdA(w)=+\infty.
\end{align*}

Before giving the other characterization of Carleson type measures on $[0,1)$, we need to  recall some results.

The following result is Lemma 1 in \cite{Mer}, which generalizes Lemma 3.1 in \cite{GM} from $p=2$ to $1<p<\infty$.

\begin{otherl}\label{inter labda B}
Let $f\in H(\D)$ with $f(z)=\sum^\infty_{n=0}a_n z^n$. Suppose  $1<p<\infty$ and  the sequence $\{a_n\}$ is a decreasing sequence of nonnegative numbers. If  $X$ is a subspace of $H(\D)$ with $\Lambda^p_{1/p}\subseteq X\subseteq\B$, then
$$
	f\in X \iff a_n=O\left(\frac1n\right).
$$
\end{otherl}

We  recall a characterization  of $s$-Carleson measure $\mu$ on [0, 1) as follows (cf. \cite[Theorem 2.1]{BW} or \cite[Proposition 1]{CGP}).
\begin{otherp}\label{u n^s}
Let $\mu$ be a finite positive Borel measure on [0, 1) and $s>0$. Then $\mu$ is an $s$-Carleson measure if and only if the sequence of moments
$\{\mu_n\}_{n=0}^\infty$ satisfies $\sup_{n\geq 0} (1+n)^s \mu_n<\infty$.
\end{otherp}

The following characterization of functions with nonnegative Taylor coefficients in $\Q_p$ is  Theorem 2.3 in  \cite{AGW}.

\begin{otherth}\label{Qp coeff}
 Let $0<p<\infty$ and let $f(z)=\sum_{n=0}^\infty a_nz^n$ be an analytic function in $\D$ with $a_n\geq 0$ for all $n$. Then
$f\in \Q_p$ if and only if
$$
\sup_{0\leq r<1} \sum_{n=0}^\infty \f{(1-r)^p}{(n+1)^{p+1}}\left(\sum_{k=0}^n(k+1)a_{k+1}(n-k+1)^{p-1}r^{n-k}\right)^2<\infty.
$$
\end{otherth}

We need the following   well-known estimates  (cf.  \cite[Lemma 3.10]{Zhu}).
\begin{otherl}\label{useful  estimates}
 Let $\beta$ be any real number.  Then
$$
\int^{2\pi}_0\frac{d\theta}{|1-ze^{-i\theta}|^{1+\beta}}\thickapprox
\begin{cases}1 & \enspace \text{if} \ \ \beta<0,\\
                     \log\frac{2}{1-|z|^2} & \enspace  \text{if} \ \  \beta=0,\\
                     \frac{1}{(1-|z|^2)^\beta} & \enspace \text{if}\ \  \beta>0,
                   \end{cases}
$$
for all $z\in \D$.
\end{otherl}

For $0<s<\infty$ and    a finite positive Borel
measure $\mu$ on $[0, 1)$, set
$$
f_{\mu, s}(z)=\sum_{n=0}^\infty \frac{\Gamma(n+s)}{\Gamma(s)n!} \mu_n z^n, \ \ z\in \D.
$$
Now we state the other main result in this section     which is inspired by Lemma \ref{inter labda B} and Proposition \ref{u n^s}. %For $s>1$, there exists measures $\mu$ such that $\left\{\frac{\Gamma(n+s)}{\Gamma(s)n!} u_n\right\}$ is not decreasing. The proof given below is  without using Lemma \ref{inter labda B} and Proposition \ref{u n^s}.

\begin{prop} \label{newCM2}
Suppose  $0<s<\infty$ and   $\mu$ is a finite positive Borel
measure on $[0, 1)$. Let $1<p<\infty$ and let  $X$ be  a subspace of $H(\D)$ with $\Lambda^p_{1/p}\subseteq X\subseteq\B$.
Then $\mu$ is an $s$-Carleson measure  if and only if
$f_{\mu, s}\in X$.
\end{prop}

\begin{proof}
 Let $\mu$ be  an $s$-Carleson measure. Clearly,
 $$
 f_{\mu, s}(z)=\int_{[0, 1)} \frac{1}{(1-tz)^s}d\mu(t)
 $$
 for any $z\in \D$.
  For $p>1$, it follows from  the Minkowski inequality and Lemma \ref{useful  estimates} that
\begin{align*}
M_p(r, f'_{\mu, s})\leq& s \left(\frac{1}{2\pi}\int_0^{2\pi} \left(\int_{[0, 1)}\frac{1}{|1-tre^{i\theta}|^{s+1}}d\mu(t)\right)^pd\theta\right)^{1/p}\\
\leq & s  \int_{[0, 1)} \left( \frac{1}{2\pi}\int_0^{2\pi} \frac{1}{|1-tre^{i\theta}|^{(s+1)^p}}d\theta \right)^{1/p}  d\mu(t)\\
\lesssim& \int_{[0, 1)} \frac{1}{(1-tr)^{s+1-\frac{1}{p}}} d\mu(t)
\end{align*}
for all $0<r<1$. Combining this with Proposition \ref{newCM1}, we get $f_{\mu, s}\in \Lambda^p_{1/p}$ and hence $f_{\mu, s}\in X$.

On the other hand, let $f_{\mu, s}\in X$. Then $f_{\mu, s}\in \Q_q$ with $q>1$. By the Stirling formula,
$$
\frac{\Gamma(n+s)}{\Gamma(s)n!}\thickapprox (n+1)^{s-1}
$$
for all nonnegative  integers $n$. Consequently, by Theorem \ref{Qp coeff} we deduce
{\small
\begin{eqnarray*}
\infty&>&\sum_{n=0}^\infty \f{(1-r)^q}{(n+1)^{q+1}}\left(\sum_{k=0}^n(k+2)^{s}\mu_{k+1}(n-k+1)^{q-1}r^{n-k}\right)^2\\
&\gtrsim& \sum_{n=0}^\infty \f{(1-r)^q}{(4n+1)^{q+1}}\left(\sum_{k=0}^{4n}(k+2)^{s}\mu_{k+1}(4n-k+1)^{q-1}r^{4n-k}\right)^2\\
&\gtrsim& \sum_{n=0}^\infty \f{(1-r)^q}{(4n+1)^{q+1}}\left(\sum_{k=n}^{2n}(k+2)^{s}\int_r^1 t^{k+1}d\mu(t) (4n-k+1)^{q-1}r^{4n-k}\right)^2\\
&\gtrsim& \mu^2([r, 1))  (1-r)^q\sum_{n=0}^\infty \f{r^{8n+2}}{(4n+1)^{q+1}}\left(\sum_{k=n}^{2n}(k+2)^{s}(4n-k+1)^{q-1} \right)^2\\
&\gtrsim&\mu^2([r, 1)) (1-r)^q \sum_{n=0}^\infty (4n+2)^{2s+q-1} r^{8n+2}\\
&\thickapprox& \f{\mu^2([r, 1))}{(1-r)^{2s}}
\end{eqnarray*}
}
for all $r\in [0, 1)$ which yields that $\mu$ is an $s$-Carleson measure.  The proof is complete.
\end{proof}

\section{$\qp$ spaces and the range of  $\Cu$ acting on $H^\infty$}

In this section, we characterize finite  positive Borel measures $\mu$ on $[0,1)$ such that $\C_\mu(H^\infty)\subseteq \qp$ for $0<p<2$. Descriptions of Carleson measures in     Proposition \ref{newCM1}  play  a key role in our proof.

The following lemma is from \cite{OF}.

\begin{otherl}\label{estiamtes}
Suppose $s>-1$, $r>0$, $t>0$ with $r+t-s-2>0$. If $r$, $t<2+s$, then
$$
\int_\D \frac{(1-|z|^2)^s}{|1-\overline{a}z|^r|1-\overline{b}z|^t}dA(z)\lesssim \frac{1}{|1-\overline{a}b|^{r+t-s-2}}
$$
for all $a$, $b\in \D$. If $t<2+s<r$, then
$$
\int_\D \frac{(1-|z|^2)^s}{|1-\overline{a}z|^r|1-\overline{b}z|^t}dA(z)\lesssim \frac{(1-|a|^2)^{2+s-r}}{|1-\overline{a}b|^{t}}
$$
for all $a$, $b\in \D$.
\end{otherl}

We give our result as follows.
\begin{theor}\label{1main}
Suppose  $0<p<2$ and $\mu$ is a finite positive Borel measure on $[0,1)$. Then  $\C_\mu(H^\infty)\subseteq \qp$ if and only if $\mu$ is a Carleson measure.
\end{theor}

\begin{proof}
Suppose $\C_\mu(H^\infty)\subseteq \qp$. Then $\C_\mu(H^\infty)$ is a subset of the Bloch space. By  \cite[Theorem 5]{GGM}, $\mu$ is a Carleson measure.

Conversely, suppose  $\mu$ is a Carleson measure and $f\in H^\infty$. Then $f$ is also in the Bloch space $\B$.  From Proposition 1 in \cite{GGM},
$$
\C_{\mu}(f)(z)=\int_{[0, 1)} \frac{f(tz)}{1-tz}d\mu(t), \ \ z\in \D.
$$
Hence for any $z\in \D$,
\begin{align}\label{31}
&\|\Cu (f)\|_{\qp} \nonumber \\
\lesssim & \sup_{a\in\D} \left(\int_{\D}\left(\int_{[0,1)}\frac{|tf'(tz)|}{|1-tz|}d\mu(t)\right)^2(1-|\sigma_a(z)|^2)^p dA(z)\right)^{\frac12} \nonumber \\
	& +\sup_{a\in\D}\left(\int_{\D}\left( \int_{[0,1)}\frac{|tf(tz)|}{|1-tz|^2}d\mu(t)\right)^2(1-|\sigma_a(z)|^2)^p dA(z)\right)^{\frac12} \nonumber \\
\lesssim & \|f\|_\B  \sup_{a\in\D} \left(\int_{\D}\left(\int_{[0,1)}\frac{1}{(1-|tz|)|1-tz|}d\mu(t)\right)^2(1-|\sigma_a(z)|^2)^p dA(z)\right)^{\frac12} \nonumber \\
	& +\|f\|_{H^\infty}\sup_{a\in\D}\left(\int_{\D}\left( \int_{[0,1)}\frac{1}{|1-tz|^2}d\mu(t)\right)^2(1-|\sigma_a(z)|^2)^p dA(z)\right)^{\frac12}.\nonumber \\
\end{align}
Let  $c$ be a positive constant such that $2c<\min\{2-p, p\}$. Then
\begin{equation}\label{32}
(1-|tz|)^2\geq (1-t)^{2-2c} (1-|z|)^{2c}
\end{equation}
 for all $t\in [0, 1)$ and all $z\in \D$.
By the Minkowski inequality, (\ref{32}), Lemma \ref{estiamtes} and Proposition \ref{newCM1},  we get
\begin{align}\label{33}
&\sup_{a\in\D} \left(\int_{\D}\left(\int_{[0,1)}\frac{1}{(1-|tz|)|1-tz|}d\mu(t)\right)^2(1-|\sigma_a(z)|^2)^p dA(z)\right)^{\frac12} \nonumber \\
\leq&\sup_{a\in\D}  \int_{[0,1)} \left( \int_{\D} \frac{1}{(1-|tz|)^2|1-tz|^2} (1-|\sigma_a(z)|^2)^p dA(z) \right)^{\frac12} d\mu(t)\nonumber \\
\lesssim &\sup_{a\in\D}(1-|a|^2)^{\frac p2}\int_{[0,1)}\frac{1}{(1-t)^{1-c}}d\mu(t)\big(\int_{\D}\frac{(1-|z|^2)^{p-2c}}{|1-tz|^2|1-\bar{a}z|^{2p}} dA(z)\big)^{\frac12} \nonumber \\
\lesssim &\sup_{a\in\D}\int_{[0,1)}\frac{(1-|a|^2)^{\frac p2}}{(1-t)^{1-c}|1-ta|^{\frac{p}{2}+c}}d\mu(t)<\infty.
\end{align}
Similarly, it follows from Lemma \ref{estiamtes} and Proposition \ref{newCM1} that
\begin{align}\label{34}
&\sup_{a\in\D}\left(\int_{\D}\left( \int_{[0,1)}\frac{1}{|1-tz|^2}d\mu(t)\right)^2(1-|\sigma_a(z)|^2)^p dA(z)\right)^{\frac12} \nonumber \\
\leq & \sup_{a\in\D} \int_{[0,1)} \left( \int_{\D} \frac{1}{|1-tz|^4} (1-|\sigma_a(z)|^2)^p dA(z)\right)^{\frac12} d\mu(t)¡¡\nonumber¡¡\\
\lesssim & \sup_{a\in\D}\int_{[0,1)}\frac{(1-|a|^2)^{\frac p2}}{(1-t^2)^{1-\frac p2}|1-at|^p}d\mu(t)<\infty.
\end{align}
From (\ref{31}), (\ref{33}) and (\ref{34}), we get that $\Cu (f)\in \qp$. The proof is complete,
\end{proof}
\noindent {\bf  Remark 3.}\ \
Set $d\mu_0(x)=dx$ on [0, 1).  Then $d\mu_0$ is a Carleson measure and $\C_{\mu_0}(1)(z)=\frac{1}{z}\log\f{1}{1-z}$.
Clearly,  the function $\C_{\mu_0}(1)$  is not in the Dirichlet space. Thus Theorem \ref{1main} does not hold when $p=0$.

Note that $\qp=\B$ for any $p>1$. Theorem \ref{1main} generalizes Theorem 5 in \cite{GGM} from the Bloch space $\B$ to all $\qp$ spaces. For $p=1$, Theorem \ref{1main} gives an answer to a question raised in \cite[p. 20]{GGM}.
The proof given here highlights   the role of Proposition \ref{newCM1}. In the next section, we give  a more general result where an alternative proof of Theorem \ref{1main} will be  provided.

\section{$s$-Carleson measures and the range of  another Ces\`aro-like operator acting on $H^\infty$ }

It is also natural to consider how  the characterization of $s$-Carleson measures in Proposition \ref{newCM2} can  play a  role in the investigation of the range of  Ces\`aro-like operators acting on $H^\infty$. We consider this
topic by  another kind of Ces\`aro-like operators.

Suppose  $0<s<\infty$ and  $\mu$ is a finite  positive Borel measure on $[0,1)$.
For $f(z)=\sum_{n=0}^\infty a_nz^n$ in  $H(\D)$, we define
$$
\C_{\mu, s} (f)(z)=\sum^\infty_{n=0}\left(\mu_n\sum^n_{k=0}\frac{\Gamma(n-k+s)}{\Gamma(s)(n-k)!}a_k\right)z^n, \quad z\in\D.
$$
  Clearly, $\C_{\mu, 1}$ is equal to $\C_{\mu}$.

\begin{limma}\label{intergera repre}
Suppose  $0<s<\infty$ and  $\mu$ is a finite  positive Borel measure on $[0,1)$. Then
$$
\C_{\mu, s} (f)(z)=\int_{[0,1)}\frac{f(tz)}{(1-tz)^s}d\mu(t)
$$
for $f\in H(\D)$.
\end{limma}
\begin{proof}
The proof follows from a simple calculation with power series. We omit it.
\end{proof}

We have the following result.

\begin{theor}\label{2main}
Suppose  $0<s<\infty$ and $\mu$ is a finite  positive  Borel measure on $[0,1)$. Let $\max\{1, \f{1}{s}\}<p<\infty$ and $X$ is a subspace of $H(\D)$ with $\Lambda^p_{1/p}\subseteq X\subseteq\B$.  Then
$\C_{\mu, s}(H^\infty)\subseteq X$ if and only if $\mu$ is an $s$-Carleson measure.
\end{theor}
\begin{proof}

Let $\C_{\mu, s}(H^\infty)\subseteq X$. Then $\C_{\mu, s}(1)\in X$; that is, $f_{\mu, s}\in X$. It follows from Proposition \ref{newCM2} that $\mu$ is an $s$-Carleson measure.

On the other hand, let $\mu$ be an $s$-Carleson measure and $f\in H^\infty$. By Lemma \ref{intergera repre}, we see
\begin{align*}
	\Cus (f)'(z)=\int_{[0,1)}\frac{tf'(tz)}{(1-tz)^s}d\mu(t)+ \int_{[0,1)}\frac{stf(tz)}{(1-tz)^{s+1}}d\mu(t), \quad z\in\D.
\end{align*}
Then
\begin{align}\label{41}
&\sup_{0<r<1}(1-r)^{1-\frac1p}\left(\f{1}{2\pi}\int_0^{2\pi}|\Cus (f)'(re^{i\theta})|^p d\theta\right)^{\frac1p}\nonumber \\
\lesssim &\|f\|_\B \sup_{0<r<1}(1-r)^{1-\frac1p} \left(\f{1}{2\pi}\int_0^{2\pi}\left(\int_{[0,1)} \frac{1}{|1-tre^{i\theta}|^{s}(1-tr)} d\mu(t)\right)^p d\theta\right)^{\frac1p} \nonumber \\
&+\|f\|_{H^\infty}\sup_{0<r<1}(1-r)^{1-\frac1p} \left(\f{1}{2\pi}\int_0^{2\pi}\left(\int_{[0,1)} \frac{1}{|1-tre^{i\theta}|^{s+1}} d\mu(t)\right)^p d\theta\right)^{\frac1p}. \nonumber  \\
\end{align}
Note that $ps>1$. By  the Minkowski inequality, Lemma \ref{useful  estimates} and Lemma \ref{S-CM}, we deduce
\begin{align}\label{42}
&\sup_{0<r<1}(1-r)^{1-\frac1p} \left(\f{1}{2\pi}\int_0^{2\pi}\left(\int_{[0,1)} \frac{1}{|1-tre^{i\theta}|^{s}(1-tr)} d\mu(t)\right)^p d\theta\right)^{\frac1p} \nonumber \\
\leq & \sup_{0<r<1}(1-r)^{1-\frac1p} \int_{[0,1)} \left(\f{1}{2\pi}\int_0^{2\pi} \frac{1}{|1-tre^{i\theta}|^{sp}(1-tr)^p}  d\theta\right)^{\frac1p}  d\mu(t) \nonumber \\
\lesssim &  \sup_{0<r<1}(1-r)^{1-\frac1p}  \int_{[0,1)} \f{1}{(1-tr)^{s+1-\frac1p}}  d\mu(t)<\infty,  \nonumber \\
\end{align}
and
\begin{align}\label{43}
&\sup_{0<r<1}(1-r)^{1-\frac1p} \left(\f{1}{2\pi}\int_0^{2\pi}\left(\int_{[0,1)} \frac{1}{|1-tre^{i\theta}|^{s+1}} d\mu(t)\right)^p d\theta\right)^{\frac1p}\nonumber \\
\lesssim & \sup_{0<r<1}(1-r)^{1-\frac1p} \int_{[0,1)} \left( \f{1}{2\pi}\int_0^{2\pi}\frac{1}{|1-tre^{i\theta}|^{(s+1)p} } d\theta  \right)^{\frac1p} d\mu(t) \nonumber \\
\lesssim & \sup_{0<r<1}(1-r)^{1-\frac1p} \int_{[0,1)} \f{1}{(1-tr)^{s+1-\frac1p}}   d\mu(t)<\infty.
\end{align}
From (\ref{41}), (\ref{42}) and (\ref{43}),  $\Cus(f)\in \Lambda^p_{1/p}$. Note that $\Lambda^p_{1/p}\subseteq X$.  The desired result follows.
\end{proof}


\begin{thebibliography}{WWW}

\bibitem{AL} R. Aulaskari and P. Lappan,  Criteria for an analytic function to be
Bloch and a harmonic or meromorphic function to be normal,  Complex
analysis and its applications, Pitman Res. Notes in Math., 305,
 Longman Sci. Tech., Harlow,
 1994, 136-146.

\bibitem{AXZ}  R. Aulaskari, J. Xiao and R. Zhao,  On subspaces and
subsets of $BMOA$ and $UBC$, {\sl  Analysis},  {\bf 15} (1995), 101-121.

\bibitem{AGW} R. Aulaskari, D. Girela and H. Wulan, Taylor coefficients and mean growth of the derivative of $\mathcal{Q}_p$ functions, {\sl J. Math. Anal. Appl.}, {\bf 258} (2001), 415-428.

\bibitem{ASX} R. Aulaskari, D. Stegenga and  J. Xiao, Some subclasses of $BMOA$ and their characterization in terms of Carleson
measures, {\sl  Rocky Mountain J. Math.}, {\bf 26} (1996),  485-506.

\bibitem{B} A. Baernstein II,  Analytic functions of bounded mean oscillation,  Aspects of Contemporary Complex Analysis, Academic Press, 1980, 3-36.

\bibitem{BW} G. Bao and H. Wulan, Hankel matrices acting on Dirichlet spaces, {\sl J. Math. Anal. Appl.},  {\bf 409} (2014),  228-235.

\bibitem{BWY} G. Bao,  H. Wulan and F. Ye, The range of the Ces\`aro operator acting on $H^\infty$, {\sl Canad. Math. Bull.}, {\bf 63} (2020),  633-642.

\bibitem{BYZ} G. Bao, F. Ye and K. Zhu, Hankel measures for Hardy spaces, {\sl  J. Geom. Anal.}, {\bf 31} (2021),   5131-5145.

\bibitem{Bla} O. Blasco, Operators on weighted Bergman spaces ($0<p\leq 1$) and applications, {\sl Duke Math. J.},  {\bf 66} (1992),  443-467.

\bibitem{BSS}  P. Bourdon, J. Shapiro and W. Sledd, Fourier series, mean Lipschitz spaces, and bounded mean oscillation, in: Analysis at Urbana, vol. I, Urbana, IL, 1986-1987, in: London Math. Soc. Lecture Note Ser., vol. 137,
1989, pp. 81-110.

\bibitem{CGP} C. Chatzifountas, D. Girela and J. Pel\'aez, A generalized Hilbert matrix acting on Hardy spaces, {\sl J. Math. Anal. Appl.},  {\bf 413} (2014), 154-168.

\bibitem {DS} N. Danikas and A.  Siskakis, The Ces\`aro operator on bounded analytic functions,  {\sl Analysis},  {\bf 13 } (1993), 295-299.

\bibitem  {D} P. Duren,   Theory of $H^p$ Spaces, Academic Press, New York, 1970.



\bibitem {EX} M. Ess\'en and J. Xiao,  Some results on $\Q_p$ spaces, $0<p<1$, {\sl J. Reine Angew. Math.},  {\bf 485} (1997), 173-195.

\bibitem {Gar} J. Garnett, Bounded analytic functions, Springer, New York, 2007.

\bibitem{GGM}  P. Galanopoulos, D. Girela and N. Merch\'an, Ces\`aro-like operators acting on spaces of analytic functions, {\sl  Anal. Math. Phys.},  {\bf 12}  (2022),  Paper No. 51.

 \bibitem{Gir} D. Girela, Analytic functions of bounded mean oscillation. In: Complex Function
Spaces, Mekrij\"arvi 1999  Editor: R. Aulaskari. Univ. Joensuu Dept. Math. Rep. Ser.,
4, Univ. Joensuu, Joensuu, (2001) pp. 61-170.

\bibitem{GM} D. Girela and  N. Merch\'an, A Hankel matrix acting on spaces of analytic functions, {\sl  Integr. Equ. Oper. Theory}, {\bf 89} (2017),  581-594.


\bibitem{Mer}  N. Merch\'an,   Mean Lipschitz spaces and a generalized Hilbert operator, {\sl  Collect. Math.}, {\bf 70} (2019),  59-69.

\bibitem {M} J. Miao, The Ces\'aro operator is bounded on $H^p$ for $0<p<1$, {\sl Proc. Amer. Math. Soc.}, {\bf 116} (1992), 1077-1079.

\bibitem{OF}  J. Ortega and J.  F\'abrega, Pointwise multipliers and corona type decomposition in $BMOA$, {\sl Ann. Inst.
Fourier (Grenoble)},  {\bf 46} (1996), 111-137.

\bibitem{S1} A. Siskakis, Composition semigroups and the Ces\'aro operator on $H^p$, {\sl J. London Math. Soc.},  {\bf 36} (1987),  153-164.

\bibitem{S2} A.  Siskakis, The Ces\'aro operator is bounded on $H^1$, {\sl Proc. Amer. Math. Soc.},  {\bf 110} (1990),  461-462.

\bibitem {X1} J. Xiao,  Carleson measure, atomic decomposition and free interpolation from Bloch space,
{\sl Ann. Acad. Sci. Fenn. Ser. A I Math.},  {\bf 19} (1994), 35-46.

\bibitem  {X2} J. Xiao,  Holomorphic $\Q$  classes, Springer, LNM 1767, Berlin, 2001.

\bibitem  {X3} J. Xiao,  Geometric $\Q_p$  functions, Birkh\"auser Verlag, Basel-Boston-Berlin, 2006.

\bibitem  {Zhu} K. Zhu,   Operator theory in function spaces, American Mathematical Society, Providence, RI, 2007.


\end{thebibliography}
\end{document}